\documentclass[11pt]{amsart}

\usepackage{graphicx}

\usepackage[all]{xy}
\usepackage{amssymb,mathrsfs,enumitem,amsmath,bbold,mathtools,url,amsthm}
\usepackage[mathscr]{euscript}

\newtheorem{theorem}{Theorem}[section]
\newtheorem{corollary}[theorem]{Corollary}

\newtheorem{proposition}[theorem]{Proposition}

\theoremstyle{definition}

\DeclareMathOperator{\Dend}{\mathsf{Dend}}
\DeclareMathOperator{\Ass}{\mathsf{Ass}}
\DeclareMathOperator{\Poisson}{\mathsf{Poisson}}
\DeclareMathOperator{\Com}{\mathsf{Com}}
\DeclareMathOperator{\PL}{\mathsf{PreLie}}
\DeclareMathOperator{\PP}{\mathsf{PrePoisson}}

\DeclareMathOperator{\Perm}{\mathsf{Perm}}
\DeclareMathOperator{\Leib}{\mathsf{Leib}}

\DeclareMathOperator{\Dias}{\mathsf{Diass}}
\DeclareMathOperator{\Lie}{\mathsf{Lie}}

\DeclareMathOperator{\Ab}{Ab}
\DeclareMathOperator{\gr}{gr}
\DeclareMathOperator{\END}{\mathbf{End}}
\DeclareMathOperator{\Mod}{Mod}
\DeclareMathOperator{\Set}{\mathbf{Set}}

\DeclareMathOperator{\birc}{\bar{\circ}}
\DeclareMathOperator{\bdot}{\bar{\cdot}}

\DeclareMathOperator{\Vect}{\mathbf{Vect}}

\DeclareMathAlphabet{\mathbbold}{U}{bbold}{m}{n}

\newcommand{\bbk}{\mathbbold{k}}

\newcommand{\scrF}{\mathscr{F}}
\newcommand{\scrG}{\mathscr{G}}

\newcommand{\scrL}{\mathscr{L}}
\newcommand{\scrM}{\mathscr{M}}
\newcommand{\scrN}{\mathscr{N}}
\newcommand{\scrP}{\mathscr{P}}

\newcommand{\scrR}{\mathscr{R}}
\newcommand{\scrX}{\mathscr{X}}

\newcommand{\sfC}{\mathsf{C}}
\newcommand{\sfD}{\mathsf{D}}
\newcommand{\sfS}{\mathsf{S}}
\newcommand{\sfT}{\mathsf{T}}

\begin{document}

\title{Endofunctors and Poincar\'e--Birkhoff--Witt theorems}

\author{Vladimir Dotsenko}

\address{Institut de Recherche Math\'ematique Avanc\'ee, UMR 7501\\ Universit\'e de Strasbourg et CNRS\\ 7 rue Ren\'e-Descartes, 67000 Strasbourg, France}
\email{vdotsenko@unistra.fr}

\author{Pedro Tamaroff}

\address{School of Mathematics, Trinity College, Dublin 2, Ireland}

\email{pedro@maths.tcd.ie}

\subjclass[2010]{16D90 (Primary), 16S30, 17B35, 18D50 (Secondary)}

\keywords{endofunctor, monad, universal enveloping algebra, Poincar\'e--Birkhoff--Witt theorem}

\begin{abstract}
We determine what appears to be the bare-bones categorical framework for Poincar\'e--Birkhoff--Witt type theorems about universal enveloping algebras of various algebraic structures. Our language is that of endofunctors; we establish that a natural transformation of monads enjoys a Poincar\'e--Birkhoff--Witt property only if that transformation makes its  codomain a free right module over its domain. We conclude with a number of applications to show how this unified approach proves various old and new Poincar\'e--Birkhoff--Witt type theorems. In particular, we prove a PBW type result for universal enveloping dendriform algebras of pre-Lie algebras, answering a question of Loday. 
\end{abstract}

\maketitle

\section{Introduction}

It is well known that the commutator $[a,b]=ab-ba$ in every associative algebra satisfies the Jacobi identity. Thus, every associative algebra may be regarded as a Lie algebra, leading to a functor from the category of associative algebras to the category of Lie algebras assigning to an associative algebra the Lie algebra with the same underlying vector space and the Lie bracket as above. This functor admits a left adjoint $U(-)$, the universal enveloping associative algebra of a Lie algebra. The classical Poincar\'e--Birkhoff--Witt (PBW) theorem identifies the underlying vector space of the universal enveloping algebra of any  Lie algebra with its symmetric algebra; the precise properties of such an identification depend on the proof one chooses.  

More generally, a functor  from the category of algebras of type $\sfS$ to the category of algebras of type~$\sfT$ is called a functor of change of structure if it only changes the structure operations, leaving the underlying object of an algebra intact. Informally, one says that such a functor has the PBW property if, for any $\sfT$-algebra $A$, the underlying object of its universal enveloping $\sfS$-algebra  $U_\sfS(A)$ admits a description that does not depend on the algebra structure, but only on the underlying object of~$A$. 

This intuitive view of the PBW property is inspired by the notion of a PBW pair of algebraic structures due to Mikhalev and Shestakov~\cite{MiSh}. There, the algebraic setup is that of varieties of algebras. The authors of~\cite{MiSh} define, for any $\sfT$-algebra~$A$, a canonical filtration on the universal enveloping algebra~$U_\sfS(A)$ which is compatible with the $\sfS$-algebra structure, and establish that there is a canonical surjection 
 \[
\pi\colon U_\sfS(\Ab A)\twoheadrightarrow\gr U_\sfS(A) , 
 \]
where $\Ab A$ is the Abelian $\sfS$-algebra on the underlying vector space of~$A$. They say that the given algebraic structures  form a PBW pair if that canonical surjection is an isomorphism. Furthermore, they prove a result stating that this property is equivalent to $U_\sfS(A)$ having a basis of certain monomials built out of the basis elements of~$A$, where the definition of monomials  does not depend on a particular algebra~$A$. This latter property is defined in a slightly more vague way than the former one; trying to formalise it, we discovered a pleasant categorical context where PBW theorems belong. The approach we propose is to use the language of endofunctors, so that a fully rigorous way to say ``the definition of monomials  does not depend on a particular algebra'' is to say that the underlying vector space of~$U_\sfS(A)$ is isomorphic to $\scrX(A)$, where $\scrX$ is an endofunctor on the category of vector spaces, with isomorphisms $U_\sfS(A)\cong\scrX(A)$ being natural with respect to algebra maps.  Our main result (Theorem~\ref{th:PBWNat}) states that if algebraic structures are encoded by monads, and a functor of change of structure arises from a natural transformation of monads $\phi\colon\scrM\to\scrN$, then the PBW property holds if and only if the right module action of~$\scrM$ on~$\scrN$ via~$\phi$ is free; moreover the space of generators of $\scrN$ as a right $\scrM$-module is naturally isomorphic to the endofunctor $\scrX$ above. 

In the context of the classical PBW theorem for Lie algebras, the condition of freeness of a module does emerge in a completely different way: when working with Lie algebras over rings, one would normally require the Lie algebra to be free as a module over the corresponding ring in order for the PBW theorem to hold, see~\cite{CNS}. We feel as though we have to emphasize that our ``freeness of a module'' condition is of entirely different nature: it is freeness of action of one monad on another, which only makes sense when one goes one level up in terms of categorical abstraction and considers all algebras of the given type as modules over the same monad. This condition is not expressible if one looks at an individual algebra, and this is precisely what makes our main result completely new in comparison with existing literature on PBW type theorems. It is also worth mentioning that one class of operads for which the free right module condition is almost tautologically true is given by those obtained by means of distributive laws \cite{Markl}; however, for many interesting examples it is definitely not the case. (The example of post-Poisson algebras in the last section of this paper should be very instructional for understanding that.)

It is worth remarking that there is a number of other phenomena which are occasionally referred to as PBW type theorems. One of them deals with various remarkable families of associative algebras depending on one or more parameters, and is completely out of our scope; we refer the reader to the survey \cite{ShWi} for further information. The other one deals with universal enveloping algebras defined as forgetful functors as above, but considers situations where the universal enveloping algebra admits what one would agree to consider a ``nice'' description. One important feature of such a ``nice'' description is what one can call a ``baby PBW theorem'' stating that the natural map from an algebra to its universal enveloping algebra is an embedding. By contrast with our result which in particular shows that the PBW property holds for all algebras if and only if it holds for free algebras, checking the baby PBW property requires digging into intricate properties of individual algebras: there exist examples of algebraic structures for which the baby PBW property holds for all free algebras but nevertheless fails for some non-free algebras.  A celebrated example where the baby PBW property holds but the full strength PBW property is not available is given by universal enveloping diassociative algebras of Leibniz algebras~\cite{Lo01}; further examples can be found in \cite{CIL,Gu2} and \cite{Bo1,Bo3}. 
 
We argue that our result, being a necessary and sufficient statement, should be regarded as \emph{the} bare-bones framework for studying the PBW property; as such, it provides one with a unified approach to numerous PBW type results proved ---sometimes by very technical methods--- in the literature, see e.g.~\cite{Ch1,CNS,Gu3,Kol1,Kol2,MoPeSh,Pe07,Ro65,Se94,Sto93}. Most of those  PBW type theorems  tend to utilise something extrinsic; e.g., in the case of Lie algebras, one may consider only Lie algebras associated to Lie groups and identify the universal enveloping algebra with the algebra of distributions on the group supported at the unit element (see~\cite{Se64}, this is probably the closest in spirit to the original proof of Poincar\'e~\cite{Po00}), or use the additional coalgebra structure on the universal enveloping algebra (like in the proof of Cartier~\cite{Car55}, generalised by Loday in~\cite{Lo08} who defined a general notion of a ``good triple of operads''). 
Proofs that do not use such \emph{deus ex machina} devices normally rely on an explicit presentation of universal enveloping algebras by generators and relations (following the most famous application of Bergman's Diamond Lemma~\cite{Be78}, in the spirit of proofs of Birkhoff~\cite{Bi37} and Witt~\cite{Wi37}); while very efficient, those proofs break functoriality in a rather drastic way, which is highly undesirable for objects defined by a universal property. 
Finally, what is often labelled as a categorical approach to the PBW theorem refers to proving the PBW theorem for Lie algebras in an arbitrary $\bbk$-linear tensor category (over a field~$\bbk$ of characteristic zero) recorded in~\cite{DM}; this approach is indeed beautifully functorial but does not at all clarify what property of the pair of algebraic structures $(\Lie,\Ass)$ makes it work. Our approach, in addition to being fully intrinsic and functorial, unravels the mystery behind that very natural question.

Towards the end of this paper, we present a few applications of our framework. In particular, we prove a new PBW theorem for universal enveloping dendriform algebras of pre-Lie algebras (Theorem~\ref{th:PreLieDend}), thus answering a question of Loday that remained open for a decade. The proof of that result demonstrates that our monadic approach to PBW type theorems opens a door for utilising a range of operadic techniques which previously were mainly used for purposes of homotopical algebra \cite{BrDo,LV}. Another application of our main result was recently obtained in \cite{Khor} where a PBW type theorem for associative universal enveloping algebras of operadic algebras is proved; a hint for importance of operadic right modules for such a statement to hold can be found in \cite[Sec.~10.2]{Fr09}.  

To conclude this introduction, it is perhaps worth noting that our definition of the PBW property exhibits an interesting ``before/after'' dualism with that of~\cite{MiSh}: that definition formalises the intuitive notion that ``operations on~$A$ do not matter before computing $U_\sfS(A)$'', so that operations on~$U_\sfS(A)$ have some canonical ``leading terms'', and then corrections that do depend on operations of~$A$, while our approach suggests that ``operations on~$A$ do not matter after computing~$U_\sfS(A)$'', so that the underlying vector space of $U_\sfS(A)$ is described in a canonical way. In Proposition~\ref{prop:CharP}, we show that our formalisation, unlike that of~\cite{MiSh}, shows that the extent to which a PBW isomorphism may be functorial depends on the characteristic of the ground field, rather than merely saying ``certain strategies of proof are not available in positive characteristic''.

\subsection*{Acknowledgements}
We thank Dmitry Kaledin and Ivan Shestakov for extremely useful and encouraging discussions of this work. These discussions happened when the first author was visiting CINVESTAV (Mexico City); he is grateful to Jacob Mostovoy for the invitation to visit and to present this work. We also thank Vsevolod Gubarev, Pavel Kolesnikov and Bruno Vallette for useful comments, and Anton Khoroshkin for informing us of the preprint~\cite{Khor} that builds upon our work. Special thanks due to Martin Hyland whose questions greatly helped to make the proof of the main result more comprehensible.

\section{Recollections: monads, algebras, modules}
 
In this section, we recall some basic definitions and results from category theory used in this paper, referring the reader to \cite{Lin69-1,Lin69-2,Mac71} for further details.
 
\subsection{Monads}
Let $\sfC$ be a category.  Recall that all endofunctors of $\sfC$ form a strict monoidal category $(\END(\sfC), \circ, \mathbbold{1})$. More precisely, in that category morphisms are natural transformations,
the monoidal structure $\circ$ is the composition of endofunctors, $(\scrF\circ\scrG)(c)=\scrF(\scrG(c))$, and the unit of the monoidal structure~$\mathbbold{1}$ is the identity functor, $\mathbbold{1}(c)=c$. A \emph{monad} on $\sfC$ is a monoid $(\scrM,\mu_\scrM,\eta_\scrM)$ in $\END(\sfC)$; here we denote by $\mu_\scrM\colon \scrM\circ\scrM\rightarrow\scrM$ the monoid product, and by $\eta_\scrM\colon\mathbbold{1}\rightarrow \scrM$ the monoid unit.

\subsection{Algebras}
An \emph{algebra for the monad $\scrM$} is an object $c$ of $\sfC$, and a structure map
 \[
\gamma_c\colon\scrM(c)\to c
 \]
for which the
two diagrams
 \[
 \xymatrix@M=6pt{
\scrM(\scrM(c))\ar@{->}^{\scrM(\gamma_c)}[rr] \ar@{->}_{\mu_{\scrM}(c)}[d] & &  \scrM(c) \ar@{->}^{\gamma_c}[d]  \\ 
\scrM(c)\ar@{->}^{\gamma_c}[rr]  & & c      
 }
\qquad 
 \xymatrix@M=6pt{
\mathbbold{1}(c)\ar@{->}^{\eta_\scrM(c)}[rr] \ar@{->}_{1_c}[drr] & &  \scrM(c) \ar@{->}^{\gamma_c}[d]  \\ 
 & & c      
 }
 \]
commute for all~$c$. The category of algebras over a monad $\scrM$ is denoted by $\sfC^\scrM$. 

\subsection{Modules}
The notion of a module over a monad follows the general definition of a module over a monoid in monoidal category. We shall primarily focus on right modules; left modules are defined similarly. A \emph{right module over a monad $\scrM$} is an endofunctor $\scrR$ together with a natural transformation 
 \[
\rho_\scrR\colon\scrR\circ\scrM\rightarrow\scrR
 \]
for which the two diagrams
 \[
 \xymatrix@M=6pt{
\scrR(\scrM(\scrM(c)))\ar@{->}^{\scrR(\mu_\scrM(c))}[rr] \ar@{->}_{\rho_{\scrR}(\scrM(c))}[d] & &  \scrR(\scrM(c)) \ar@{->}^{\rho_\scrR(c)}[d]  \\ 
\scrR(\scrM(c))\ar@{->}_{\rho_\scrR(c)}[rr]  & & \scrR(c)      
 }
\qquad 
 \xymatrix@M=6pt{
\scrR(\mathbbold{1}(c))\ar@{->}^{\scrR(\eta_\scrM(c))}[rr] \ar@{->}^{1_{\scrR(c)}}[drr]  & &  \scrR(\scrM(c)) \ar@{->}^{\rho_\scrR(c)}[d]  \\ 
  & & \scrR(c)      
 }
 \]
commute for all~$c$. The category of right modules over a monad $\scrM$ is denoted by~$\Mod_\scrM$. The forgetful functor from the category $\Mod_\scrM$ to $\END(\sfC)$ has a left adjoint, called the \emph{free right $\scrM$-module} functor; the free right $\scrM$-module generated by an endofunctor $\scrX$ is $\scrX\circ\scrM$ with the structure map $\scrX\circ\scrM\circ\scrM \to \scrX\circ\scrM$ given by~$1_\scrX\circ\mu_\scrM$. 

\subsection{Coequalizers in categories of algebras}
Recall that a \emph{reflexive pair} in a category $\sfC$ is a diagram
 \[
 \xymatrix@M=6pt{
c_1 \ar@/^1pc/^{f}[rr] \ar@/_1pc/_{g}[rr] && \ar@{->}^{d}[ll] c_2 ,
 }
 \]
where $fd=gd=1_{c_2}$. Throughout this paper, we shall assume  the following property of the category~$\sfC$:  for every monad $\scrM$, the category $\sfC^\scrM$ has coequalizers of all reflexive pairs. There are various criteria for that to happen, see, for instance, \cite{AdKo80} and \cite[Sec.~9.3]{BaWe85} (both relying on the seminal work of Linton on coequalizers in categories of algebras~\cite{Lin69}). In particular, this property holds for any complete and cocomplete well-powered regular category where all regular epimorphisms split. This holds, for instance, for the category $\Set$ and the categories $\Vect_\bbk$ (the category of vector spaces over~$\bbk$, for any field~$\bbk$) and~$\Vect_\bbk^\Sigma$  (the category of symmetric sequences over~$\bbk$, for a field~$\bbk$ of zero characteristic), as well as their ``super'' ($\mathbb{Z}$- or $\mathbb{Z}/2$-graded) versions, which are the main categories where we expect our results to be applied. 

\section{Categorical PBW theorem}

\subsection{The adjunction between change of structure and direct image}
Suppose that $\scrM$ and $\scrN$ are two monads on $\sfC$, and that $\phi\colon\scrM\rightarrow\scrN$ is a natural transformation of monads. For such data, one can define the \emph{functor of change of algebra structure} 
 \[
\phi^*\colon \sfC^\scrN\to\sfC^\scrM 
 \]
for which the algebra map $\scrM(c)\to c$ on an $\scrN$-algebra $c$ is computed as the composite 
 \[
\scrM(c)\xrightarrow{\phi(1_c)}\scrN(c)\xrightarrow{\gamma_c} c .
 \]
By \cite[Prop.~1]{Lin69}, under our assumptions on $\sfC$ the functor $\phi^*$ has a left adjoint functor, the \emph{direct image functor} $\phi_!$, and for every $\scrM$-algebra $c$,  the $\scrN$-algebra $\phi_!(c)$ can be computed as the coequalizer of the reflexive pair of morphisms
 \[
 \xymatrix@M=6pt{
\scrN(\scrM(c)) \ar^{1_\scrN(\phi(1_c))}[rr] \ar@/_1pc/_{ 1_\scrN(\gamma_c)}[rrrr] &&\scrN(\scrN(c)) \ar^{ \mu_\scrN(1_c)}[rr] &&\scrN(c) ,
 }
 \]
which is reflexive with the arrow $d\colon\scrN(c)\to\scrN(\scrM(c))$ given by
 \[
\scrN(c)\xrightarrow{\cong}\scrN(\mathbbold{1}(c))\xrightarrow{1_\scrN(\eta_\scrM(1_c))}\scrN(\scrM(c)) .
 \]

Let us give a toy example of this general construction which would be familiar to a reader without a systematic categorical background. Let $\sfC=\Vect_\bbk$ be a category of vector spaces over a field~$\bbk$, and let $A$ be an associative algebra over~$\bbk$. Consider the endofunctor $\scrM_A$ of $\Vect_\bbk$ given by $\scrM_A(V)=A\otimes V$. It is easy to see that the associative algebra structure on~$A$ leads to a monad structure on~$\scrM_A$, and algebras over the monad $\scrM_A$ are left $A$-modules. Moreover, if $\psi\colon A\to B$ is a morphism of associative algebras, we have a natural transformation of monads $\phi\colon\scrM_A\to\scrM_B$, and the functors $\phi^*$ and $\phi_!$ are the usual restriction and induction functors between the categories of left modules. 

\smallskip 

In general, the direct image functor is well understood and frequently used in the case of analytic endofunctors~\cite{Jo85}, i.~e. in the case of operads~\cite{LV}; in that case this formula for the adjoint functor fits into the general framework of relative composite products of operadic bimodules~\cite{GaJo,Re96}. Relative products of arbitrary endofunctors do not, in general, satisfy all the properties of relative composite products; however, in some situations all the necessary coequalizers exist (and are absolute); as a consequence, for our purposes there is no need to restrict oneself to analytic endofunctors.

\subsection{The main result}
As we remarked above, our goal is to give a categorical formalisation of an intuitive view of the PBW property according to which ``the underlying object of the universal enveloping algebra of $c$ does not depend on the algebra structure of~$c$''. Suppose that $\phi\colon\scrM\rightarrow\scrN$ is a natural transformation of monads on $\sfC$. We shall say that the datum $(\scrM,\scrN,\phi)$ \emph{has the PBW property} if there exists an endofunctor~$\scrX$ such that the underlying object of the universal enveloping $\scrN$-algebra $\phi_!(c)$ of any $\scrM$-algebra~$c$ is isomorphic to~$\scrX(c)$ naturally with respect to morphisms in $\sfC^\scrM$. Using this definition, one arrives at a very simple and elegant formulation of the PBW theorem. Note that using the natural transformation $\phi$, we can regard $\scrN$ as a right $\scrM$-module via the maps
$\scrN\circ\scrM\xrightarrow[1_\scrN\circ\phi]{}\scrN\circ\scrN\xrightarrow[\mu_\scrN]{}\scrN$.

\begin{theorem}\label{th:PBWNat}
Let  $\phi\colon\scrM\rightarrow\scrN$ be a natural transformation of monads. The datum $(\scrM,\scrN,\phi)$ has the PBW property if and only if the right $\scrM$-module action on $\scrN$ via $\phi$ is free.
\end{theorem}

\begin{proof}

Let us first suppose that the datum $(\scrM,\scrN,\phi)$ has the PBW property, and let $\scrX$ be the corresponding endofunctor. Let us take an object $d$ of $\sfC$ and consider the free $\scrM$-algebra $c=\scrM(d)$; we shall now show that the direct image $\phi_!(c)$ is the free $\scrN$-algebra $\scrN(d)$. To that end, we note that there is an obvious commutative diagram 
 \[
 \xymatrix@M=4pt{
\sfC^\scrN \ar[dr] \ar^{\phi^*}[rr]&& \ar[dl]\sfC^\scrM\\
&\sfC&
 }
 \]
where the arrows to $\sfC$ are obvious forgetful functors from the categories of algebras. All the three functors in this commutative diagram are right adjoint functors, so the corresponding diagram of the left adjoint functors also commutes, meaning that free $\scrM$-algebras are sent under $\phi_!$ to free $\scrN$-algebras: we have $\phi_!(\scrM(d))\cong\scrN(d)$ naturally in~$d$. Combining this result with the PBW property, we see that we have a natural isomorphism 
 \[
\scrN(d)\cong\phi_!(\scrM(d))\cong\scrX(\scrM(d))=(\scrX\circ\scrM)(d) ,
 \]
which shows that $\scrN\cong \scrX\circ\scrM$ on the level of endofunctors. Finally, we note that the pair of arrows
 \[
 \xymatrix@M=6pt{
\scrN(\scrM(\scrM(d))) \ar^{1_\scrN(\phi(1_{\scrM(d)}))}[rr] \ar@/_1pc/_{ 1_\scrN(\gamma_{\scrM(d)})}[rrrr] &&\scrN(\scrN(\scrM(d))) \ar^{ \mu_\scrN(1_{\scrM(d)})}[rr] &&\scrN(\scrM(d)) .
 }
 \]
that defines $\phi_!(\scrM(d))$ as a coequaliser arises from evaluating the diagram 
 \[
 \xymatrix@M=6pt{
\scrN\circ \scrM\circ \scrM \ar^{1_\scrN\circ \phi \circ 1_{\scrM}}[rr] \ar@/_1pc/_{ 1_\scrN\circ \gamma_{\scrM}}[rrrr] &&\scrN\circ \scrN\circ \scrM \ar^{ \mu_\scrN\circ 1_{\scrM}}[rr] &&\scrN\circ \scrM .
 }
 \]
of right $\scrM$-modules and their maps on the object $c$. This shows that the isomorphism of endofunctors we obtained agrees with the right module action, and hence $\scrN$ is a free right $\scrM$-module. 

\smallskip 

The other way round, suppose that $\scrN$ is a free right $\scrM$-module, so that $\scrN\cong\scrX\circ\scrM$ for some endofunctor~$\scrX$. To prove that the datum $(\scrM,\scrN,\phi)$ has the PBW property, we shall utilize a very well known useful observation: in any \emph{split fork} diagram
 \[
 \xymatrix@M=6pt{
c_1   \ar@/^1pc/^{f}[rr] \ar@/_1pc/_{g}[rr]  && \ar_{t}[ll] c_2 \ar@/^0.5pc/^{e}[rr] && \ar@/^0.5pc/^{s}[ll] d 
 }
 \] 
where $es=1_d$, $ft=1_{c_2}$, and $gt=se$, $d$ is the coequalizer of the pair $f,g$. 

The $\scrN$-algebra $\phi_!(c)$ is the coequalizer of the reflexive pair
 \[
 \xymatrix@M=6pt{
\scrN(\scrM(c)) \ar^{1_\scrN(\phi(1_c))}[rr] \ar@/_1pc/_{ 1_\scrN(\gamma_c)}[rrrr] &&\scrN(\scrN(c)) \ar^{ \mu_\scrN(1_c)}[rr] &&\scrN(c) ,
 }
 \]
Note that the composition of the arrows $\scrN(\scrM(c))\xrightarrow{1_\scrN(\phi(1_c))} \scrN(\scrN(c))\xrightarrow{\mu_\scrN(1_c)}\scrN(c)$ is the definition of the right module action of $\scrM$ on $\scrN$, so under the isomorphism of right modules $\scrN\cong\scrX\circ\scrM$, the above  pair of arrows becomes
 \[
 \xymatrix@M=6pt{
\scrX(\scrM(\scrM(c))) \ar@/^0.5pc/^{\quad 1_\scrX(\mu_\scrM(1_c))}[rr] \ar@/_0.5pc/_{\quad  1_{\scrX\circ\scrM}(\gamma_c)}[rr]  &&\scrX(\scrM(c)) ,
 }
 \]
Let us prove that $\phi_!(c)\cong\scrX(c)$ by demonstrating that this pair of arrows can be completed to a split fork with $\scrX(c)$ as the handle of the fork. To that end, we define the arrow $e\colon \scrX(\scrM(c))\to\scrX(c)$ to be $1_\scrX(\gamma_c)$, the arrow $s\colon \scrX(c)\to\scrX(\scrM(c))$ to be the composite 
 \[
\scrX(c) \xrightarrow{\cong} \scrX(\mathbbold{1}(c)) \xrightarrow{1_\scrX(\eta_\scrM(1_c)) }\scrX(\scrM(c)) ,
 \]
and the arrow $t\colon\scrX(\scrM(c))\to\scrX(\scrM(\scrM(c)))$ to be the composite 
 \[
\scrX(\scrM(c)) \xrightarrow{\cong} \scrX(\mathbbold{1}(\scrM(c)))\xrightarrow{1_\scrX(\eta_\scrM(1_{\scrM(c)})) } \scrX(\scrM(\scrM(c))) ,
 \]
so the property $es=1_{\scrX(\scrM(c))}$ follows from the unit axiom for the algebra $c$, the property $ft=1_{\scrX(\scrM(c))}$ follows from the unit axiom for the monad $\scrM$, and also $se=gt$ by a direct inspection. This verification was natural in $c$ with respect to morphisms in~$\sfC^\scrM$,  so we have $\phi_!(c)\cong\scrX(c)$ naturally in $c$, and the datum $(\scrM,\scrN,\phi)$ has the PBW property.
\end{proof}

Continuing with the toy example of endofunctors $\scrM_A$ of $\Vect_\bbk$, freeness of $\scrM_B$ as as a right $\scrM_A$-module corresponds (at least for augmented algebras) to freeness of $B$ as a right $A$-module. If we have a right $A$-module isomorphism $B\cong X\otimes A$, the underlying space of the induced module $B\otimes_A M$ is isomorphic to $X\otimes M$, and does not depend on the module structure on~$M$. For instance, this is frequently used in representation theory to obtain an explicit description for the underlying spaces of induced representations of groups and of Lie algebras; in the latter case freeness follows from the classical PBW theorem. Our result offers another PBW-flavoured viewpoint for such an explicit description. 

\section{Case of analytic endofunctors}

Most interesting instances where our results have so far found applications deal with the case where the endofunctors $\scrM$ and~$\scrN$ are analytic~\cite{Jo85}, so that the monads are in fact operads~\cite{LV}. In this section, we shall mainly discuss the case $\sfC=\Vect_\bbk$, where $\bbk$ is a field of characteristic zero. In general, for analytic endofunctors to make sense and satisfy various familiar properties, it is enough to require that the category~$\sfC$ is symmetric monoidal cocomplete (including the hypothesis that the monoidal structure distributes over colimits). To state and prove a homological criterion for freeness like the one in Section~\ref{sec:Homol}, one has to make some extra assumptions, e.g. assume that the category of symmetric sequences~$\sfC^\Sigma$ is a concrete Abelian category where epimorphisms split. 

\subsection{Homological criterion of freeness}\label{sec:Homol}

We begin with setting up our main technical tool, a homological criterion of freeness of right modules.  It is well known that operadic right modules are generally easier to work with than left modules, since the composite product of analytic endofunctors is linear in the first argument. In particular, one has the wealth of homological algebra constructions that are applicable to the Abelian category of right modules, see~\cite{Fr09} for details. Moreover, for connected weight graded operads over a field of characteristic zero, one can define the notion of a minimal free resolution of a weight graded module and prove its existence and uniqueness up to isomorphism, like it is done for modules over rings in the seminal paper of Eilenberg~\cite{Ei56}. This leads to a homological criterion for freeness of a right $\scrM$-module~$\scrR$.

Recall that for an operad $\scrM$, its left module $\scrL$, and its right module $\scrR$, there is a two-sided bar construction $\mathsf{B}_\bullet(\scrR,\scrM,\scrL)$. In somewhat concrete terms, it is spanned by rooted trees where for each tree the root vertex is decorated by an element of $\scrM$, the internal vertices whose all children are leaves are decorated by elements of $\scrN$, and other internal vertices are decorated by elements of $\scrP$; the differential contracts edges of the tree and uses the operadic composition and the module action maps. For an operad with unit, this bar construction is acyclic; moreover, for a connected weight graded operad $\scrM$ the two-sided bar construction $\mathsf{B}_\bullet(\scrM,\overline{\scrM},\scrM)$ is acylic. This leads to a free resolution of any right-module $\scrR$ as 
$$\scrR\circ_\scrM\mathsf{B}_\bullet(\scrM,\overline{\scrM},\scrM)\cong\mathsf{B}_\bullet(\scrR,\overline{\scrM},\scrM).$$
This resolution can be used to prove the following result.

\begin{proposition}\label{prop:Homol}
Let $\scrM$ be a connected weight graded operad acting on $\Vect_\bbk$, and let $\scrR$ be a weight graded right $\scrM$-module. The right module $\scrR$ is free if and only if the positive degree homology of the bar construction $\mathsf{B}_\bullet(\scrR,\overline{\scrM},\mathbbold{1})$ vanishes; in the latter case, $\scrR$ is generated by $H_0(\mathsf{B}_\bullet(\scrR,\scrM,\mathbbold{1}))$. 
\end{proposition}

\begin{proof}
This immediately follows from the existence and uniqueness of the minimal free right $\scrM$-module resolution of~$\scrR$.
\end{proof}

This result is usually applied in one of the following ways. First, one can define a filtration on~$\scrR$ that is compatible with the right $\scrM$-action, and prove freeness of the associated graded module, which then by a spectral sequence argument proves freeness of~$\scrR$.
Second, one may apply the forgetful functor from symmetric operads to shuffle operads, and prove freeness in the shuffle category; since the forgetful functor is monoidal and does not change the underlying vector spaces, this guarantees vanishing of homology in the symmetric category; this approach was introduced by the first author in~\cite{Dot09}. 

\subsection{Aspects of the classical PBW theorem}

Let us first discuss how the classical Poincar\'e--Birkhoff--Witt theorem for Lie algebras fits in our framework. For that, we consider the morphism of operads $\phi\colon \Lie\to\Ass$ which is defined on generators by the formula 
 $
[a_1,a_2]\mapsto a_1\cdot a_2-a_2\cdot a_1 .  
 $ 
\subsubsection*{Case of a field of zero characteristic. } As a first step, let us outline a proof of (a version of) the classical PBW theorem (Poincar\'e~\cite{Po00}, Birkhoff~\cite{Bi37}, Witt~\cite{Wi37}) over a field $\bbk$ of characteristic zero.

\begin{theorem}\label{th:PBW}
Let $L$ be a Lie algebra over a field $\bbk$ of characteristic zero. There is a vector space isomorphism
 \[
U(L)\cong S(L)
 \] 
which is natural with respect to Lie algebra morphisms. Here $S(L)$, as usual, denotes the space of symmetric tensors in $L$.
\end{theorem}

\begin{proof}
According to Theorem \ref{th:PBWNat}, it is sufficient to establish freeness of the associative operad as a right $\Lie$-module. For that, one argues as follows. There is a filtration on the operad $\Ass$ by powers of the two-sided ideal generated by the Lie bracket $a_1\cdot a_2-a_2\cdot a_1$. The associated graded operad $\gr\Ass$ is easily seen to be generated by two operations that together satisfy the defining relations of the operad~$\Poisson$ encoding Poisson algebras and, possibly, some other relations. It is well known that for the operad $\Poisson$, we have $\Poisson\cong\Com\circ\Lie$ on the level of endofunctors, so it is a free right $\Lie$-module with generators $\Com$. 
By a straightforward computation with generating functions of dimensions, this implies that $\dim\Poisson(n)=n!=\dim\Ass(n)$, and consequently  there can be no other relations. 
By a spectral sequence argument, it is enough to prove the homology vanishing required by Proposition~\ref{prop:Homol} for the associated graded operad, so the $\Lie$-freeness of~$\Poisson$ implies the $\Lie$-freeness of~$\Ass$, with the same generators~$\Com$. Noting that $\Com(L)=S(L)$ completes the proof.  
\end{proof}

\subsubsection*{Non-functoriality of PBW in positive characteristic. }
A useful feature of the example of the morphism $\Lie\to\Ass$ is that it highlights a slight difference between our approach and the one of \cite{MiSh}. It turns out that by talking about PBW pairs, one does not detect an important distinction between the case of a field of characteristic zero and a field of positive characteristic; more precisely, the following result holds. (As the proof of Theorem~\ref{th:PBW} shows, in the characteristic zero case, such issues do not arise, and the two approaches are essentially equivalent.)

\begin{proposition}\label{prop:CharP}
Let the ground field $\bbk$ be of characteristic $p>0$. Then the pair of operads $(\Ass,\Lie)$ is a PBW-pair in the sense of \cite{MiSh}, so that $S(L)=U(\Ab L)\cong \gr U(L)$ for any Lie algebra $L$, but there is no way to choose vector space isomorphisms $S(L)\cong U(L)$ to be natural in~$L$. 
\end{proposition}

\begin{proof}
The previous argument shows that $\gr\Ass\cong\Poisson$ over any field $\bbk$. This easily implies that the canonical surjection $\pi$ is an isomorphism, establishing the PBW pair property. However, if we had vector space isomorphisms $S(L)\cong U(L)$ which are functorial in~$L$, then by Theorem~\ref{th:PBWNat} we would have $\Ass\cong\Com\circ\Lie$ as analytic endofunctors, and as a consequence the trivial submodule of $\Ass(n)\cong\bbk S_n$ would split as a direct summand, which is false in positive characteristic. 
\end{proof}

To have a better intuition about the second part of the proof, one may note that the proof of equivalence of two definitions in \cite{MiSh} goes by saying that if we have a PBW pair of algebraic structures, then, first, the universal enveloping algebra of an Abelian algebra has a basis of monomials which does not depend on a particular algebra, and then derive the same for any algebra using the PBW property. The latter step requires making arbitrary choices of liftings that cannot be promoted to an endofunctor.

\subsection{Enlarging the category of algebra objects}\label{sec:Enlarge}

Let us record a very simple corollary of Theorem~\ref{th:PBWNat} for the case of operads.  

\begin{proposition}\label{prop:ChangeCat}
Let $\phi\colon\scrM\rightarrow\scrN$ be a morphism of augmented operads that are analytic endofunctors of a category~$\sfC$. Assume that the datum $(\scrM,\scrN,\phi)$ has the PBW property, and let $\sfD$ be a category of which the category~$\sfC$ is a full subcategory. Then the datum $(\scrM,\scrN,\phi)$ has the PBW property if $\scrM$ and $\scrN$ are regarded as analytic endofunctors of~$\sfD$.
\end{proposition}

\begin{proof}
The only remark to make is that for a free module $\scrX\circ\scrM$ the space of generators $\scrX$ can be recovered as the quotient by the right action of the augmentation ideal, hence the space of generators is also an analytic endofunctor of $\sfC$. An analytic endofunctor of $\sfC$ gives rise to an analytic endofunctor of~$\sfD$, and freeness for the enlarged category follows. 
\end{proof}

As a first application of this result, the PBW property for the morphism of operads $\Lie\to\Ass$ over a field $\bbk$ of characteristic $0$ implies that the same holds for associative algebras and Lie algebras in various symmetric monoidal categories that extend the category $\Vect_\bbk$; for example, this implies that the PBW theorem for Lie superalgebras (proved in~\cite{Ro65} and re-discovered in~\cite{CNS}) and the PBW theorem for twisted Lie algebras~\cite{Sto93} do not need to be proved separately, as already indicated by Bernstein's proof of the PBW theorem~\cite{DM} mentioned in the introduction.  

A slightly less obvious application for the same morphism of operads $\Lie\to\Ass$ is to the so called Leibniz algebras, the celebrated ``noncommutative version of Lie algebras''~\cite{Bl65}. Recall that a Leibniz algebra is a vector space with a bilinear operation $[-,-]$ without any symmetries satisfying the identity $[a_1,[a_2,a_3]]=[[a_1,a_2],a_3]-[[a_1,a_3],a_2]$. For a Leibniz algebra $L$, the space $L^2$ spanned by all squares~$[x,x]$ is easily seen to be an ideal, and the quotient $L/L^2$ has a natural structure of a Lie algebra. Moreover, it is known that the quotient map $L\to L/L^2$ is a Lie algebra in the symmetric monoidal ``category of linear maps'' of Loday and Pirashvili~\cite{LoPi}. In this category, by the classical PBW theorem, the underlying object of the universal enveloping algebra of $L\to L/L^2$ is isomorphic to 
 \[
S(L\to L/L^2)\cong \left(S(L/L^2)\otimes L\to S(L/L^2)\right) .
 \]
This gives a conceptual categorical explanation of appearance of the vector space $S(L/L^2)\otimes L$ in the context of universal enveloping algebras of Leibniz algebras~\cite[Th.~2.9]{LoPi93}. 

\subsection{The PBW non-theorem for Leibniz algebras}\label{sec:Dias} A well known instance where the direct image functor $\phi_!$ can be computed explicitly but depends on the algebra structure is the case of the morphism $\Leib\to\Dias$ from the aforementioned operad of Leibniz algebras to the symmetric operad of diassociative algebras. Here diassociative algebras refer to the algebraic structure introduced by Loday~\cite{Lo01} for the purpose of studying periodicity phenomena in algebraic K-theory; a diassociative algebra is a vector space with two bilinear operations $\vdash$ and $\dashv$ satisfying the identities
\begin{gather*}
(a_1\dashv a_2)\dashv a_3=a_1\dashv(a_2\dashv a_3),\quad  (a_1\dashv a_2)\dashv a_3=a_1\dashv(a_2\vdash a_3),\\
(a_1\vdash a_2)\dashv a_3=a_1\vdash(a_2\dashv a_3),\\
(a_1\dashv a_2)\vdash a_3=a_1\vdash(a_2\vdash a_3),\quad  (a_1\vdash a_2)\vdash a_3=a_1\vdash(a_2\vdash a_3) .
\end{gather*}
The morphism $\phi\colon\Leib\to\Dias$ is defined by the formula $\phi([a_1,a_2])=a_1\dashv a_2-a_2\vdash a_1$. In fact, this pair of operads and the morphism between them come from the morphism $\Lie\to\Ass$ via a certain endofunctor of the category of operads, the tensor product with the operad usually denoted by $\Perm$, see~\cite{Ch01}. It is known~\cite{Go01} that the universal enveloping diassociative algebra of a Leibniz algebra~$L$ is, as a vector space, isomorphic to the tensor product $S(L/L^2)\otimes L$ mentioned above, and hence very much depends on the Leibniz algebra structure of~$L$. (As we saw in Section~\ref{sec:Enlarge}, it happens because Leibniz algebras give rise to Lie algebras in a larger category where $L/L^2$ is included as a part of the object.) It is natural to ask what exactly breaks in this case if one attempts to mimic our proof of the classical PBW theorem. The associated graded operad of $\Dias$ with respect to the filtration defined by the Leibniz operation is easily seen to be generated by an operation $a_1,a_2\mapsto a_1\cdot a_2$ satisfying the identities of the operad $\Perm$ and an operation $a_1,a_2\mapsto [a_1,a_2]$ satisfying the Leibniz identity; these operations are related by several identities including 
 \[
[a_1, a_2\cdot a_3]=[a_1,a_2]\cdot a_3+[a_1,a_3]\cdot a_2 \quad\text{and}\quad [a_1\cdot a_2,a_3]=[a_1,a_3]\cdot a_2-a_1\cdot[a_3,a_2] .
 \]
Expanding the operadic monomial $[a_1\cdot a_2,a_3\cdot a_4]$ in two different ways, one obtains the identity
 \[
a_1\cdot [a_2,a_4]\cdot a_3+a_1\cdot[a_4,a_2]\cdot a_3=0 ,
 \]
showing that the right $\Leib$-module is not free, and that the obstruction to freeness does indeed arise from the symmetric part of the Leibniz bracket (that vanishes on the Lie level). This identity can be lifted to a slightly less appealing identity in the operad $\Dias$, which we do not include here. 

\subsection{Universal enveloping pre-Lie algebras of Lie algebras}

In a little known paper \cite{Se94}, a PBW type theorem is proved for universal enveloping pre-Lie algebras of Lie algebras. Let us explain how this result fits into our formalism. We denote the operad encoding pre-Lie algebras by~$\PL$.  It is well known that there exists a morphism of operads $\phi\colon\Lie\to\PL$ defined by $\phi([a_1,a_2])=a_1\cdot a_2-a_2\cdot a_1$.

\begin{proposition}\label{prop:LiePreLie}
The datum $(\Lie,\PL,\phi)$ has the PBW property.
\end{proposition}

\begin{proof}
We shall once again utilise the filtration argument, considering the filtration of the operad $\PL$ by powers of the two-sided ideal generated by the Lie bracket. In \cite{Dot17}, the associated graded operad was studied. Examining the proof of the main result of~\cite{Dot17}, we see that the associated graded operad is free as a right $\Lie$-module, since that proof exhibits an explicit basis of tree monomials in the associated graded operad, and the shape of those monomials allows to apply an argument identical to that of \cite[Th.~4(2)]{Dot09}. A standard spectral sequence argument completes the proof.
\end{proof}

It is interesting that $\PL$ is also free as a left $\Lie$-module, which was used in~\cite{Ch10} to establish that for a free pre-Lie algebra $L$, the result of change of algebra structure $\phi^*(L)$ is free as a Lie algebra.

\subsection{A new PBW theorem: solution to an open problem of Loday}

We conclude this paper with new PBW type result answering a question that Jean-Louis Loday asked the first author around~2009. Namely, the operad $\Dend$ of dendriform algebras admits a morphism from the operad $\PL$, which we shall recall below. It has been an open problem to prove a  PBW-type for dendriform universal enveloping algebras of pre-Lie algebras, which we do in this section. Since this paper was produced, an alternative proof (however without functoriality) was obtained by Gubarev \cite{Gu0}. In the same paper \cite{Gu0}, some PBW-type results involving post-Lie algebras are proved; their functorial versions are obtained, using rewriting theory for shuffle operads, in a separate note by the first author \cite{Dot19a}.

Recall that the dendriform operad~$\Dend$ is the operad with two binary generators denoted by $\prec$ and~$\succ$ that satisfy the identities
\begin{gather*}
(a_1\prec a_2)\prec a_3=a_1\prec(a_2\prec a_3+a_2\succ a_3),\\
(a_1\succ a_2)\prec a_3=a_1\succ(a_2\prec a_3),\\
(a_1\prec a_2+a_1\succ a_2)\succ a_3=a_1\succ(a_2\succ a_3).
\end{gather*}
In this section, we shall consider a different presentation of the operad~$\Dend$ via the operations 
 \[
a_1\circ a_2=a_1\prec a_2-a_2\succ a_1 \quad\text{ and } \quad a_1\cdot a_2=a_1\prec a_2+a_2\succ a_1. 
  \]
By a direct computation, all relations between these operations are consequences of the identities
\begin{gather*}
(a_1\circ a_2)\circ a_3-a_1\circ(a_2\circ a_3)=(a_1\circ a_3)\circ a_2-a_1\circ(a_3\circ a_2),\\
(a_1\cdot a_2)\cdot a_3=a_1\cdot (a_2\cdot a_3)+a_1\cdot(a_3\cdot a_2)-(a_1\circ a_3)\circ a_2,\\
(a_1\cdot a_2)\circ a_3=(a_1\circ a_3)\cdot a_2-a_1\cdot(a_2\circ a_3)+a_1\cdot (a_3\circ a_2),\\
(a_1\circ a_2)\cdot a_3+(a_1\circ a_3)\cdot a_2=a_1\circ (a_2\cdot a_3)+a_1\circ(a_3\cdot a_2).
\end{gather*}
In particular, this implies the well known statement that the operation $a_1\circ a_2=a_1\prec a_2-a_2\succ a_1$ satisfies the pre-Lie identity, so that there is a morphism $\phi\colon \PL\to\Dend$ sending the generator of $\PL$ to $a_1\prec a_2-a_2\succ a_1 $. We can now state the promised new PBW theorem. 

\begin{theorem}\label{th:PreLieDend}
The datum $(\PL,\Dend,\phi)$ has the PBW property.
\end{theorem}
 
\begin{proof}
The proof of this theorem utilises the operad $\PP$ controlling pre-Poisson algebras \cite{Ag00} which we shall recall below. Since by the operad $\PL$ we always mean the operad controlling the \emph{right} pre-Lie algebras, we shall work with \emph{right} pre-Poisson algebras (opposite of those in~\cite{Ag00}).

\smallskip 

For the first step of the proof, we consider the filtration $F^\bullet\Dend$ of the operad $\Dend$ by powers of the two-sided ideal generated by the operation $a_1\circ a_2$. In the associated graded operad, the relations determined above become
\begin{gather*}
(a_1\circ a_2)\circ a_3-a_1\circ(a_2\circ a_3)=(a_1\circ a_3)\circ a_2-a_1\circ(a_3\circ a_2),\label{eq:PP1}\\
(a_1\cdot a_2)\cdot a_3=a_1\cdot (a_2\cdot a_3)+a_1\cdot(a_3\cdot a_2),\label{eq:PP2}\\
(a_1\cdot a_2)\circ a_3=(a_1\circ a_3)\cdot a_2-a_1\cdot(a_2\circ a_3)+a_1\cdot (a_3\circ a_2),\label{eq:PP3}\\
(a_1\circ a_2)\cdot a_3+(a_1\circ a_3)\cdot a_2=a_1\circ (a_2\cdot a_3)+a_1\circ(a_3\cdot a_2).\label{eq:PP4}
\end{gather*}
These are precisely the defining relations of the operad controlling right pre-Poisson algebras. Thus, the associated graded operad $\gr_F\Dend$ admits a surjective map from the operad $\PP$; this result is in agreement with \cite[Sec.~4]{Ag00} where it is shown that for a filtered dendriform algebra whose associated graded algebra is a Zinbiel algebra, that associated graded acquires a canonical pre-Poisson structure.

\smallskip 

We shall now look at the shuffle operad~$\PP^f$ associated to the operad $\PP$ via the usual forgetful functor~\cite{BrDo,DK}. It is generated by four elements $\cdot$, $\circ$, $\bdot$, $\birc$ which are the two operations and their opposites. We consider the ordering which is the superposition of the quantum monomial ordering \cite[Sec. 2]{Dot19b} for which every degree two monomial with $\cdot$ or $\bdot$ at the root and $\circ$ or $\birc$ at the non-root vertex is smaller than every degree two monomial with $\circ$ or $\birc$ at the root and $\cdot$ or $\bdot$ at the non-root vertex, and the path-lexicographic ordering induced by the ordering $\cdot<\bdot<\birc<\circ$. A slightly tedious computation shows for this choice of ordering this operad has a quadratic Gr\"obner basis; moreover, we have $\dim\PP^f(4)=336$. The surjection mentioned above leads to a surjection of vector spaces 
 \[
\PP^f(4)\twoheadrightarrow\Dend^f(4) ,
 \]
and if we note that $\dim\Dend^f(4)=4!\cdot 14=336$, we conclude that this surjection must be an isomorphism. In particular, when we pass from the operad $\Dend$ to its associated graded, no new cubic relations arise in the associated graded case (our operads are generated by binary operations,  so cubic elements live in arity~$4$). Repeating \emph{mutatis mutandis} the argument of \cite[Th.~7.1]{PP}, we see that the operad $\gr_F\Dend$ is quadratic, and $\PP\cong\gr_F\Dend$. 

\smallskip 

By direct inspection of our Gr\"obner basis of the operad~$\PP^f$, \cite[Th.~4(2)]{Dot09} applies, showing that this operad is free as a right $\PL^f$-module. By Proposition~\ref{prop:Homol} and a spectral sequence argument, the same is true for the operad~$\Dend$. 
\end{proof}

\begin{corollary}
The operad of pre-Poisson algebras is Koszul. 
\end{corollary}

\begin{proof}
This follows immediately from the fact that the associated shuffle operad has a quadratic Gr\"obner basis.
\end{proof}

A similar method for proving Koszulness works for the appropriately-defined operad of pre-Gerstenhaber algebras and for its versions with generators of degrees $a$ and $b$. This fills a gap in the literature on homotopy algebras: in \cite{AAC17} and \cite{Al2015}, the notions of pre-Gerstenhaber algebras up to homotopy and pre-$(a,b)$-algebras up to homotopy were introduced, as algebras over the cobar construction of the Koszul dual cooperad. Such a definition only makes sense if one knows that the operads in question are Koszul, which is not checked in those papers. Fortunately, it turns out to be true, as our results indicate.

\end{document}